\documentclass{article}
\usepackage{amsfonts}
\usepackage{amsmath}
\usepackage{amssymb}
\usepackage[textwidth=16cm, textheight=22cm]{geometry}
\newtheorem{theorem}{Theorem}[section]

\newtheorem{lemma}[theorem]{Lemma}
\newtheorem{proposition}[theorem]{Proposition}

\newenvironment{proof}[1][Proof]{\noindent\textbf{#1.} }{\ \rule{0.5em}{0.5em}}
\newcommand{\eproof}{}

\newcommand{\spann}{\operatorname{span}_{\oplus}}

\newcommand{\R}{\mathbb{R}}
\newcommand{\Rn}{\R^n}
\newcommand{\Rp}{\R_+}
\newcommand{\Rpn}{\Rp^n}
\newcommand{\Rpnn}{\Rp^{n\times n}}

\newcommand{\Rpmm}{\Rp^{m\times m}}

\newcommand{\kroneker}{\delta}

\newcommand{\Lin}{\operatorname{L}}
\newcommand{\lin}{\operatorname{Lin}}
\newcommand{\MB}{\operatorname{MB}}
\newcommand{\ri}{\operatorname{ri}}
\newcommand{\intererer}{\operatorname{int}}
\newcommand{\ext}{\operatorname{ext}_{\oplus}}
\newcommand{\extconv}{\operatorname{ext}}
\newcommand{\tchk}{\cdot}

\newcommand{\diag}{\operatorname{diag}}

\begin{document}
\title{On visualization scaling, subeigenvectors and Kleene stars in max algebra
\thanks{This research was supported by EPSRC grant RRAH12809, RFBR grant 08-01-00601 and joint
RFBR/CNRS grant 05-01-02807}}
\author{Serge\u{\i} Sergeev\thanks{%
School of Mathematics, University of Birmingham, Edgbaston,
Birmingham B15
2TT, UK.}
\and Hans Schneider\thanks{%
Department of Mathematics, University of Wisconsin-Madison, Madison,
Wisconsin 53706, USA.}
\and Peter Butkovi\v{c}\thanks{%
School of Mathematics, University of Birmingham, Edgbaston,
Birmingham B15 2TT, UK.}} \maketitle

\begin{abstract}
\noindent The purpose of this paper is to investigate the interplay
arising between max algebra, convexity and scaling problems. The
latter, which have been studied in nonnegative matrix theory, are
strongly related to max algebra. One problem is that of strict
visualization scaling, defined as,  for a given nonnegative matrix
$A$, a diagonal matrix $X$ such that all elements of $X^{-1}AX$ are
less than or equal to the maximum cycle geometric mean of $A$, with
strict inequality for the entries which do not lie on critical
cycles. In this paper such scalings are described by means of the
max algebraic subeigenvectors and Kleene stars of nonnegative
matrices as well as by some concepts of convex geometry.



\end{abstract}

\fontsize{14pt}{18pt} \selectfont

\section{Introduction}

The purpose of this paper is to investigate the interplay arising
between max algebra, convexity and matrix scaling. A nonnegative
matrix $A$ is called \textit{visualized} if  all its elements are
less than or equal to the maximum cycle geometric mean $\lambda(A)$
of $A$, and it is called \textit{strictly visualized} if, further,
there is strict inequality for the entries which do not lie on
critical cycles. Given a nonnegative matrix $A$, the chief aim of
this paper is to identify and characterize in several ways diagonal
matrices $X$ with a positive diagonal for which $X^{-1}AX$ is
strictly visualized, see Theorems 3.3, 3.7, 4.2 and 4.4.

In Section 2, we revisit and
appropriately summarize the theory of max algebraic eigenvectors and
subeigenvectors, and some properties of Kleene stars.

Sections 3 and 4 contain our principal results. In Section 4 our
chief tool is the Kleene star $A^*$ of $A$ (defined for a definite
matrix), and the max algebraic cone $V^*(A)$. The latter consists of
the subeigenvectors of $A$ for the eigenvalue $\lambda(A)$ or,
equivalently, of the eigenvectors of $A^*$. We call $V^*(A)$ the
\textit{subeigencone} of $A$.  It is also a convex cone. Diagonal
matrices $X$ corresponding to vectors $x$ in its relative interior
of the subeigencone are precisely the matrices $X$ that strictly
visualize $A$, see Theorem 3.7. Among those vectors $x$ are all
linear combinations of the columns of $A^*$ with positive
coefficients, see Theorem 3.3.

While in Section 3 our approach is convex geometric, the main idea
of Section 4 is to start with a strictly visualized matrix and to
describe all strict visualizers in matrix theoretic terms, see
Theorem 4.2. We also show that the dimension of the linear hull of
the subeigencone $V^*(A)$ equals the number of components of the
critical graph of the Kleene star $A^*$, see Theorem 4.4. At the end
of the section we show by example that the max algebraic dimension
of $V^*(A)$ may exceed its linear algebraic dimension.

The interplay between max algebra
(essentially equivalent to tropical algebra) and convexity, here
explored via visualization, is also important for tropical
convexity, see the papers \cite{DS-04,Jos-05, JK-08}, among many
others. We also note that visualization scalings can be important
for max algebra, due to the connections with the theory of $0-1$ matrices
that they provide. See \cite{ED-99,ED-01,Ser-09} for
recent developments and applications of this idea.

\section{Eigenvectors and subeigenvectors}


By \textit{max algebra} we understand the analogue of linear algebra
developed over the max-times semiring $\R_{\max,\times}$ which is
the set of nonnegative numbers $\R_+$ equipped with the operations
of ``addition'' $a\oplus b:=\max(a,b)$ and the ordinary
multiplication $a\otimes b:=a\times b$. The operations of the
semiring are extended to the nonnegative matrices and vectors in the
same way as in conventional linear algebra. That is if $A=(a_{ij})$,
$B=(b_{ij})$ and $C=(c_{ij})$ are matrices of compatible sizes with
entries from $\R_+$,
we write $C=A\oplus B$ if $c_{ij}=a_{ij}\oplus b_{ij}$ for all $i,j$ and $%
C=A\otimes B$ if $c_{ij}=\sum_{k}^{\oplus }a_{ik}
b_{kj}=\max_{k}(a_{ik}b_{kj})$ for all $i,j$. If $\alpha \in\Rp$
then $\alpha A=\left( \alpha a_{ij}\right) $. We assume everywhere
in this paper that $n\geq 1$ is an integer. $P_{n}$ will stand for
the set of permutations of the set $\{1,...,n\},$ and the sets like
$\{1,\ldots,m\}$ or $\{1,\ldots,n\}$ will be denoted by $[m]$ or
$[n]$, respectively. If $A$ is an $n\times n$ matrix then the
iterated product $A\otimes A\otimes
...\otimes A$ in which the symbol $A$ appears $k$ times will be denoted by $%
A^{k}$.

Max algebra is often presented in  settings which seem to be
different from $\R_{\max,\times}$, namely, over the max-plus
semiring $\R_{\max,+}=(\R\cup\{-\infty\}, \oplus=\max,\otimes=+)$
and the min-plus (or tropical) semiring
$\R_{\min,+}=(\R\cup\{+\infty\}, \oplus=\min,\otimes=+)$. The
semirings are isomorphic to each other and to $\R_{\max,\times}$. In
particular, $x\mapsto\exp(x)$ yields an isomorphism between
$\R_{\max,+}$ and $\R_{\max,\times}$.

Let $A=(a_{ij})\in\Rpnn$. The max algebraic
eigenproblem consists in finding $\lambda\in\Rp$ and
$x\in\Rpn\backslash\{0\}$ such that $A\otimes x=\lambda x$. If this
equation is satisfied, then $\lambda$ is called a {\em max algebraic
eigenvalue} of $A$ and $x$ is called a {\em max algebraic
eigenvector} of $A$ associated with the eigenvalue $\lambda$.


We will also be interested in the {\em max algebraic subeigenvectors}
associated with $\lambda$, that is, $x\in\Rpn$ such that $A\otimes
x\leq\lambda x$. Their first appearance in max algebra seems to be
\cite{Gau:92} Ch. IV and \cite{Gau-95}. For a more recent
reference, see generalization of the max-plus spectral theory \cite{AGW-05},
where they are called super-eigenvectors.

Next we explain two notions important for both the eigenproblem and the
subeigenproblem: that of the maximum cycle mean and that of the Kleene star.

Let $A=(a_{ij})\in\Rpnn$. The weighted digraph $D_A=(N(A),E(A))$,
with the set of nodes $N(A)=[n]$ and the set of edges $E(A)=N(A)\times N(A)$
with weights $w(i,j)=a_{ij}$, is called the {\em digraph
associated} with $A$.
Suppose that $\pi =(i_{1},...,i_{p})$ is a path in $D_A$, then the \textit{%
weight} of $\pi $ is defined to be $w(\pi
,A)=a_{i_{1}i_{2}}a_{i_{2}i_{3}}\ldots a_{i_{p-1}i_{p}}$ if $p>1$, and $1$ if $p=1$.
If $i_1=i_p$ then $\pi$ is called a cycle.
A path $\pi$ is called {\em positive} if $w(\pi,A)>0$.
A path which begins at $i$ and ends at $j$ will be called an {\em $i\to j$ path}.
The \textit{maximum cycle geometric mean} of $A$, further denoted by $\lambda(A)$,
is defined by the formula
\begin{equation*}
\lambda (A)=\max_{\sigma }\mu (\sigma ,A),  \label{mcm}
\end{equation*}%
where the maximization is taken over all cycles in the digraph and
\begin{equation*}
\mu (\sigma ,A)=w(\sigma ,A)^{1/k}  \label{cm}
\end{equation*}%
denotes the \textit{geometric mean} of the cycle $\sigma =(i_{1},...,i_{k},i_{1})$.

If the series $I\oplus A\oplus A^2\oplus\ldots$ converges to a finite matrix, then this matrix is called
the {\em Kleene star} of $A$ and denoted by $A^*=(a^*_{ij})$.
The next proposition
gives a necessary and sufficient condition for a matrix to be a Kleene star.

\begin{proposition}
\label{char-kleene}
\cite{BCOQ} Let $A=(a_{ij})\in\Rpnn$. The following are equivalent:
\begin{itemize}
\item[1.] $A$ is a Kleene star;
\item[2.] $A^*=A$;
\item[3.] $A^2=A$ and $a_{ii}=1$ for all $i=1,\ldots,n$.
\end{itemize}
\end{proposition}

The next theorem  explains some of the interplay between the maximum
cycle geometric mean $\lambda(A)$, the Kleene star $A^*$, and the
max algebraic eigenproblem.

\begin{theorem}
\label{mcm-kleene}
\cite{BCOQ,Bap-98,Car-71,CG:79,Vor-67} Let $A\in\Rpnn$. Then
\begin{itemize}
\item[1.] the series $I\oplus A\oplus A^2\oplus\ldots$
converges to a finite matrix $A^*$ if and only if $\lambda(A)\leq 1$, and then
$A^*=I\oplus A\oplus A^2\oplus\ldots\oplus A^{n-1}$ and $\lambda(A^*)=1$;
\item[2.] $\lambda(A)$ is the greatest max algebraic eigenvalue of $A$.
\end{itemize}
\end{theorem}

This theorem shows great similarity between max algebra and nonnegative linear algebra. However, it
also reveals a crucial difference: the series $I\oplus A\oplus A^2\oplus A^3\oplus\ldots$ converges
also if $\lambda(A)=1$.

$A\in\Rpnn$ is called {\em irreducible} if for any nodes $i$ and $j$ in
$D_A$ a positive $i\to j$ path exists.

\begin{proposition}
\label{kls-pos}
\cite{BCOQ, CG:79}
If $A$ is irreducible and $\lambda(A)\leq 1$, then $A^*$ has all entries positive.
\end{proposition}

More generally, it is important that Kleene stars accumulate the paths with greatest weights.
Namely, if $i\neq j$ then $a^*_{ij}=\max w(\pi,A)$ where $\pi$ ranges over paths from $i$ to $j$.

Matrices with $\lambda(A)=1$ are called {\em definite}.


Results involving  a Kleene star $A^*$ will be stated for definite
matrices.  There is no real loss of generality here in the case of
matrices $A$ with $\lambda(A) >0$. Indeed, for any such $A$ we have
that $\lambda(\alpha A)=\alpha\lambda(A)$, and if $\alpha>0$, then
any eigenvector of $A$ associated with $\lambda(A)$ is also an
eigenvector of $\alpha A$ associated with $\lambda(\alpha A)$ and
conversely. Hence if $\lambda(A)>0$, then the eigenproblems for $A$
and $A/\lambda(A)$, which is definite, are equivalent.

Note that $\lambda(A)=0$ implies that $A$ contains a zero column,
and then eigenvectors and subeigenvectors are just vectors $x$
satisfying $x_i=0$ whenever the corresponding column $A_{\cdot
i}\neq 0$. In what follows,
we will not treat this trivial case and we will always assume that $\lambda(A)>0$.

The spaces that we consider in max algebra are subsets of $\Rpn$
closed under componentwise maximization $\oplus$,
and scalar multiplication. They are called {\em max cones}, due to
the apparent analogy and important connections with conventionally
convex cones in $\Rpn$.

The set of subeigenvectors of $A$ associated with $\lambda(A)$ will
be denoted by $V^*(A)$. The set of eigenvectors associated with
$\lambda(A)$ will be denoted by $V(A)$. Both sets are max cones, and
hence $V(A)$ will be called the {\em eigencone} of $A$, and $V^*(A)$
will be called the {\em subeigencone} of $A$. Next we study some
simple relations between $V(A)$ and $V^*(A)$.  The first one
is immediate.

\begin{proposition}
\label{eiginsubeig}
$V(A)\subseteq V^*(A)$.
\end{proposition}

Further we denote by $\spann(A)$ the {\em max algebraic column span}
of $A$, which is the set of {\em max combinations}
$\{\sum^{\oplus}_i \lambda_i A_{\cdot i},\ \lambda_i\in\Rp\}$ of the
columns of $A$. Note that $V(A)\subseteq\spann(A)$ for any matrix
$A$.

\begin{proposition}
\label{subeigstar}
If $A$ is definite, then $V^*(A)=V(A^*)=V^*(A^*)=\spann(A^*)$.
\end{proposition}
\begin{proof}
First note that by Theorem \ref{mcm-kleene}, if $\lambda(A)=1$ then $A^*$ exists and $\lambda(A^*)=1$. Now
we show that $V^*(A)=V(A^*)$. Suppose that $A^*\otimes x=x$,
then $A\otimes x\leq x$, because $A\leq A^*$. If $A\otimes x\leq x$, then $(I\oplus A)\otimes x=x$
and also $A^*\otimes x=x$, since $A^m\otimes x\leq x$ for any $m$ (due to the monotonicity of matrix
multiplication). As $(A^*)^*=A^*$ by Prop. \ref{char-kleene}, we also have that
$V^*(A^*)=V(A^*)$.

We show that $V^*(A)=\spann(A^*)$. As $A\otimes A^*\leq A^*$, each
column of $A^*$ is a subeigenvector of $A$, hence
$\spann(A^*)\subseteq V^*(A)$. The converse inclusion follows from
$V^*(A)=V(A^*)$ and the inclusion $V(A^*)\subseteq\spann(A^*)$.
\end{proof}

A matrix $A$ will be called {\em strongly definite}, if it is definite
and if all its diagonal entries equal $1$.
Note that any Kleene star is strongly definite by Prop. \ref{char-kleene}.

\begin{proposition}
\label{eig=subeig}
For $A$ a strongly definite matrix, $V(A)= V^*(A)$.
\end{proposition}
\begin{proof}
To establish $V(A)=V^*(A)$, it is enough to show $V^*(A)\subseteq V(A)$, as the converse
inclusion is trivially true. Take $y\in V^*(A)$. We have that
$\sum^{\oplus}_{j\neq i} a_{ij}y_j\oplus y_i\leq y_i$ which is equivalent to
$\sum^{\oplus}_{j\neq i} a_{ij} y_j\oplus y_i=y_i$, so $y\in V(A)$.
\end{proof}\eproof

By the above propositions, the subeigenvectors of $A$, and in the
strongly definite case also the eigenvectors of $A$, are described
as the vectors from the max algebraic column span of $A^*$, which we
call {\em Kleene cone}.

More generally, a set $S$ is called a {\em generating set} for a max
cone $K$, written $K=\spann(S)$, if every vector $y\in K$ can be
expressed as a max combination $y={\sum_{i=1}^m}^{\oplus} \lambda_i
x^i$  of some elements $x^1,\ldots, x^m\in S$, with $\lambda_i\geq 0$ for $i\in[m]$.
A set $S$ is called a
{\em (weak) basis} for $K$ if $\spann(S)=K$ and none of the vectors
in $S$ can be expressed as a max combination of the other vectors in
$S$. A vector $y\in K$ is called a {\em max extremal} of $K$, if
$y=u\oplus w,\ u,w\in K$ implies that $y=u$ or $y=w$. The set of
max extremals $u$ of $K$  scaled with respect to the max norm, which
means that $||u||=\max_i u_i=1$, will be denoted by $\ext(K)$.
We have the following general result describing max extremals of
closed max cones.

\begin{theorem}
\label{mink} \cite{BSS-07,GK-07} If $K\subseteq\Rpn$ is a closed max cone, then the set
$\ext(K)$ is non-empty and it is the unique scaled basis for $K$.
\end{theorem}

If $K=\spann(A)$  for some matrix $A$, then $K$ is closed, so the
set $\ext(\spann(A))$ denoted by $\ext(A)$ for brevity, is non-empty
and constitutes the unique scaled basis for $\spann(A)$. In this
case the vectors of $\ext(A)$ are some of the columns of $A$ scaled
with respect to the max norm.

Next we describe the eigencone and the subeigencone of $A\in\Rpnn$, and the
sets of their scaled max extremals, in the case $\lambda(A)>0$. For this
we will need the following notions and notation.
The cycles with the cycle geometric mean equal to $\lambda(A)$ are
called {\em critical}, and the nodes and the edges of $D_A$ that
belong to critical cycles are called {\em critical}. The set of
critical nodes is denoted by $N_c(A)$, the set of critical
edges is denoted by $E_c(A)$, and the {\em critical digraph} of
$A$, further denoted by $C(A)=(N_c(A),E_c(A))$, is the digraph which
consists of all critical nodes and critical edges of $D_A$. All
cycles of $C(A)$ are critical \cite{BCOQ}. The set of nodes that are
not critical is denoted by $\overline{N_c(A)}$. By $C^*(A)$ we
denote the digraph with the set of nodes $[n]$ and the set of edges
$E_c^*(A)$ containing all the loops $(i,i)$ for $i\in [n]$ and such
that $(i,j)\in E_c^*(A)$, for $i\neq j$, if and only if there exists
an $i\to j$ path $(i_1,\ldots,i_p)$ in $C(A)$. The following theorem describes
both subeigencone and eigencone in the case when $A$ is
definite. For two vectors $x$ and $y$, we write $x\sim y$ if
$x=\lambda y$ for $\lambda>0$.

\begin{theorem}
\label{subeigs} Let $A\in\Rpnn$ be a definite matrix, and let $M(A)$ denote
a fixed set of indices such that for each strongly connected
component of $C(A)$ there is a unique index of
that component in $M(A)$.
Then $A^*$ is strongly definite, and
\begin{itemize}
\item[1.] the following are equivalent: $(i,j)\in E_c(A)$,
$a_{ij}a^*_{jk}=a^*_{ik}$ for all $k\in[n]$, $a^*_{kj}=a^*_{ki}a_{ij}$ for all $k\in[n]$.
\item[2.] the following are equivalent: $(i,j)\in E_c^*(A)$, $A^*_{\cdot i}\sim A^*_{\cdot j}$,
$A^*_{i\cdot}\sim A^*_{j\cdot}$;
\item[3.] any column of $A^*$ is a max extremal of $\spann(A^*)$;
\item[4.] $V(A)$ is described by
\begin{equation*}
V(A) =\left\{ {\sum_{i\in M(A)}}^\oplus \lambda_i A^*_{\tchk i};\ \lambda_i\in\Rp\right\},
\end{equation*}
and $\ext(V(A))$ is the set of scaled columns of $A^*$ whose indices
belong to $M(A)$;
\item[5.] for any $y\in V^*(A)$ and any $(i,j)\in E_c(A)$ we have $a_{ij} y_j=y_i$;
\item[6.] $V^*(A)$ is described by
\begin{equation*}
V^*(A)=V(A^*) = \left\{ {\sum_{i\in M(A)}}^\oplus \lambda_i A^*_{\tchk i}\oplus\sum_{j\in\overline{N_c(A)}}^{\oplus}
\lambda_j A^*_{\tchk j};\ \lambda_i,\lambda_j\in\Rp\right\},
\end{equation*}
and $\ext(V^*(A))=\ext(A^*)$ is the set of scaled columns of $A^*$
whose indices belong to $M(A)\cup\overline{N_c(A)}$.
\end{itemize}
\end{theorem}
\begin{proof}
Statements 1.-4. are well-known \cite{BCOQ, CG:79, CG:95, Gau:92, HOW:05}.

We show 5.: By Prop. \ref{subeigstar}, any $y\in V^*(A)$ is a max
combination of the columns of $A^*$. Let $(i,j)\in E_c(A)$, then part
1. implies that $a_{ij}z_j=z_i$ for any $z=A^*_{\cdot k}$,
$k\in[n]$. As $y$ is a max combination of all these, it follows that
$a_{ij}y_j=y_i$.

We show 6.: By Prop. \ref{subeigstar} we have $V^*(A)=\spann(A^*)$
and any column of $A^*$ is a max extremal of $\spann(A^*)$ by part 3.
By 2. we have that $A^*_{\cdot i}\sim A^*_{\cdot j}$ if and only if
$(i,j)\in E_c^*(A)$, hence all the columns in $M(A)$ are independent max
extremals and any other columns with indices in $N_c(A)$ are
proportional to them. Also note that there are no edges $(i,j)\in
E_c^*(A)$ such that $i\notin N_c(A)$ or $j\notin N_c(A)$ except for
the loops, and therefore all columns in $\overline{N_c(A)}$ are also
independent max extremals.
\end{proof}

The number of connected components of $C(A)$ will be denoted by
$n(C(A))$. For a finitely generated max cone $K$ the cardinality of
its unique scaled basis will be called the {\em max algebraic
dimension} of $K$.
Parts 4. and 6.
of Theorem \ref{subeigs} yield the following corollary.

\begin{proposition}
\label{maxdim} For any matrix $A\in\Rpnn$ with $\lambda(A)>0$ we
have that the max algebraic dimension of $V(A)$ is equal to
$n(C(A))$, and the max algebraic dimension of $V^*(A)$ is equal to
$n(C(A))+|\overline{N_c(A)}|$.
\end{proposition}

For $x\in\Rpn$ denote by $\diag(x)$ the diagonal matrix with entries
$\delta_{ij}x_i$, for $i,j\in[n]$, where $\delta_{ij}$ is the
Kronecker symbol (that is, $\delta_{ij}=1$ if $i=j$ and
$\delta_{ij}=0$ if $i\neq j$). Note that the max algebraic
multiplication by a diagonal matrix is not different from the
conventional multiplication, and therefore the notation $\otimes$
will be omitted in this case. If $x$ is positive, then $X=\diag(x)$
is invertible both in max algebra and in the ordinary linear
algebra, and the inverse $X^{-1}$ has entries $\delta_{ij}x_i^{-1}$,
for $i,j\in[n]$. The spectral properties of a matrix $A$ do not
change significantly if we apply a diagonal similarity scaling
$A\mapsto X^{-1}AX$, where $X=\diag(x)$, with a positive $x\in\Rpn$.


The following proposition follows very easily from results in the
diagonal scaling literature, see e.g. Remark 2.9 of \cite{ES}

\begin{proposition}
\label{diag-sim-eq}
Let $A\in\Rpnn$ and let $B=X^{-1}AX$, where $X=\diag(x)$, with positive $x\in\Rpn.$
Then
\begin{itemize}
\item[1.] $w(\sigma
,A)=w(\sigma ,B)$ for every cycle $\sigma$, hence $\lambda
(A)=\lambda (B)$ and $C(A)=C(B)$;
\item[2.] $V(A)=\{Xy\mid y\in V(B)\}$ and $V^*(A)=\{Xy\mid y\in V^*(B)\}$
\item[3.] $A$ is definite if and only if $B$ is definite, and in this case
$B^*=X^{-1}A^*X$.
\end{itemize}
\end{proposition}


\section{Subeigenvectors, visualization and convexity}

We call $x\in\Rpn$ a {\em nonnegative linear combination} (resp.
a {\em log-convex combination}) of $y^1,\ldots,y^m\in\Rpn$, if $x=\sum_{i=1}^m \lambda_i y^i$
with $\lambda_i\geq 0$ (resp. $x=\prod_{i=1}^m (y^i)^{\lambda_i}$ with $\lambda_i\geq 0$
and $\sum_{i=1}^m \lambda_i=1$, and both power and multiplication taken componentwise).
The combinations are called {\em positive} if $\lambda_i>0$ for all $i$.
A set $K\subseteq\Rpn$ is called a {\em convex cone} (resp. a {\em log-convex set}),
if it is stable under linear combinations (resp. under log-convex combinations).

In max arithmetics, $a\oplus b\leq c$ is equivalent to
$a\leq c$ and $b\leq c$. Using this, one can write out a
system of very special homogeneous linear inequalities which define
the subeigencone of $A$, and 
hence this cone is also a convex cone and a log-convex set.

\begin{proposition}
\label{subeig-conv}
Let $A\in\Rpnn$ and $\lambda(A)>0$. Then $V^*(A)$ is a max cone, a convex cone
and a log-convex set.
\end{proposition}
\begin{proof}
We have that
\begin{equation*}
\label{v*a1}
\begin{split}
V^*(A)&=\{y\mid A\otimes y\leq\lambda(A)y\}=\{y\mid
{\sum_j}^{\oplus} a_{ij} y_j\leq
\lambda(A)y_i\; \forall i\}=\\
&=\{y\mid a_{ij} y_j\leq\lambda(A) y_i\;\forall i,j\}.
\end{split}
\end{equation*}
Each set $\{y\mid a_{ij} y_j\leq\lambda(A)y_i\}$ is a max cone, a convex cone and
a log-convex set, hence the same
is true about $V^*(A)$, which is the intersection of these sets.
\end{proof}\eproof

The log-convexity in $(\R_+\backslash\{0\})^n$ (i.e. in the max-times
setting) corresponds to the conventional convexity in $\R^n$
(i.e., the max-plus setting or the min-plus setting). We also note
that $\{y\mid a_{ij}y_j\leq\lambda(A) y_i\}$ and hence $V^*(A)$ are
closed under some other operations. In particular, $V^*(A)$ is closed under
componentwise $p$-norms $\oplus_p$ defined by $(y\oplus_p z)_i=(y_i^p+z_i^p)^{1/p}$
for $p>0$.

Prop. \ref{subeig-conv} raises a question whether or not there exist max cones
containing positive vectors,
which are finitely generated and convex, other than Kleene cones.
The results of \cite{JK-08} suggest that the answer is negative.

Let $K$ be a convex cone, then $y\in K$ is called an {\em extremal}
of $K$ if and only if $y=\lambda u+\mu v$, where $u,v\in K$, implies
$y\sim u$ (and hence also $y\sim v$). The set of scaled extremals of $K$
will be denoted by $\extconv(K)$.

\begin{proposition}
Let $A\in\Rpnn$ and $\lambda(A)>0$, then $\ext(V^*(A))\subseteq\extconv(V^*(A))$.
\end{proposition}
\begin{proof}
Without loss of generality we assume that $A$ is definite.
By Theorem \ref{subeigs} part 6., $\ext(V^*(A))$ is the set of
scaled columns of $A^*$, after eliminating the repetitions. As
$a^*_{ik} a_{kk}^*=a_{ik}^*$, for all $i,k\in[n]$, we have that the
$x:=A^*_{\cdot k}$ satisfies $a_{ik}^* x_k=x_i$ for all $i\in[n]$.
As $V^*(A)=V^*(A^*)$ by Proposition \ref{subeigstar}, we have that
$a_{ik}^* z_k\leq z_i$ for any $z\in V^*(A)$ and all $i\in[n]$,
implying that if $x=\lambda z^1+\mu z^2$ with $z^1,z^2\in V^*(A)$,
then $a^*_{ik}z_k^s=z_i^s$ for all $i\in[n]$ and $s=1,2$. Hence
$z^1\sim x$ and $z^2\sim x$ meaning that $x\in\extconv(V^*(A))$.
\end{proof}

We note that the convex extremals $\extconv(V^*(A))$ correspond to the
pseudovertices of tropical polytropes
\cite{JK-08} (Kleene cones in the min-plus setting),
and it is known that the number of these may be up to
$(2(n-1))!/(n-1)!$ \cite{DS-04,JK-08}, unlike the number of max extremals
$\ext(V^*(A))$ which is not more than $n$.

Max algebraic subeigenvectors give rise to useful diagonal
similarity scalings. A matrix $A$ is called {\em visualized} (resp.
{\em strictly visualized}), if $a_{ij}=\lambda(A)$ for all $(i,j)\in
E_c(A)$, and $a_{ij}\leq\lambda(A)$ for all $(i,j)\notin E_c(A)$
(resp. $a_{ij}<\lambda(A)$ for all $(i,j)\notin E_c(A)$).

In the context of max algebra, visualizations have been used to
obtain better bounds on the convergence of the power method
\cite{ED-99,ED-01}. Strong links between diagonal scaling and max
algebra were established in \cite{BS}.

Specifically, Corollary 2.9 of \cite{BS} shows that for a definite
$A\in\Rpnn$, $X^{-1}AX$ is visualized if and only if $X = \diag(x)$
where $x$ is nonnegative linear combination of the columns of $A^*$
that is positive.

 Strict visualization was treated in a special case
\cite{But-03}, in connection with the strong regularity of
max-plus matrices.

A preliminary version of the following theorem appeared in
\cite{BS-07}.

\begin{theorem}
\label{fiedler-ptak1}
Let $A\in\Rpnn$ be definite and $X=\diag(x)$ with positive $x\in\Rpn$.
Then $X^{-1}AX$ is strictly visualized if any of the following
conditions are true:
\begin{itemize}
\item[1.] $x$ is a positive linear combination of
all columns of $A^*$;
\item[2.] $A$ is irreducible and $x$ is a positive
log-convex combination of all columns of $A^*$.
\end{itemize}
\end{theorem}
\begin{proof}
The following argument goes for both cases. In both cases, $x$ is
positive: for positive linear combinations this is true since
$a^*_{ii}=1$ for all $i$, and for positive log-convex combinations,
Prop.~\ref{kls-pos} assures that $A^*$ is positive if $A$ is
irreducible. As $x\in V^*(A)$, we have that $a_{ij}x_j\leq x_i$ for
all $i,j$. By Theorem~\ref{subeigs} part~5., $a_{ij}x_j=x_i$ for
all $(i,j)\in E_c(A)$. If $(i,j)\notin E_c(A)$, then, by
Theorem~\ref{subeigs} part~1., $a_{ij}z_j<z_i$ for $z=A^*_{\cdot
i}$, while $a_{ij}z_j\leq z_i$ for all $z=A^*_{\cdot k}$ where $k\in[n]$. After
summing these inequalities for all $z=A^*_{\cdot k}$ with positive
coefficients, or after raising them in positive powers and
multiplying, we obtain that $a_{ij}x_j<x_i$, taken into account the
strict inequality for $z=A^*_{\cdot i}$. Thus $x$ is positive,
$x_i^{-1}a_{ij}x_j=1$ for all $(i,j)\in E_c(A)$ and $x_i^{-1}a_{ij}
x_j<1$ for all $(i,j)\notin E_c(A)$.
\end{proof}\eproof

Note that if $A$ is definite, then every column of $A^*$ can be used to obtain a visualization
of $A$, which may not be strict. This result was known to 
Afriat \cite{A:63,A:74} and Fiedler and Pt\'{a}k \cite{FP},
and it has been a source of inspiration for many works
on scaling problems, see \cite{ES,ES:75,HS,RSS:92,SS:90,SS:91}.

Theorem \ref{fiedler-ptak1} implies the following.

\begin{proposition}
\label{fiedler-ptak2}
Let $A$ have $\lambda(A)>0$, then there exists $X=\diag(x)$ with positive $x\in\Rpn$
such that $X^{-1}AX$ is strictly visualized.
\end{proposition}

If $A$ is definite and irreducible then $A^*$ is irreducible, and in
this case $A^*$ has an essentially unique positive linear algebraic
eigenvector, called the {\em Perron eigenvector} \cite{BP}. As it is
a positive linear combination of the columns of $A^*$, we have the
following.

\begin{proposition}
\label{Perron}
Let $A\in\Rpnn$ be definite and irreducible and let $x$ be the Perron eigenvector
of $A^*$. Then $X^{-1}AX$, for $X=\diag(x)$, is strictly visualized.
\end{proposition}

We will now give a topological description
of strict visualization scalings, using the
linear hull and relative interior of $V^*(A)$.

By Theorem \ref{subeigs} part 5., for all $y\in V^*(A)$ and $(i,j)\in E_c(A)$ we have
$a_{ij} y_j=y_i$. This can be formulated geometrically.
For $A\subseteq\Rpn$ consider the set
\begin{equation*}
\label{eq:pca}
\Lin(C(A))=\{x\in\Rn\mid a_{ij} x_j=\lambda(A)x_i\; \forall (i,j)\in E_c(A)\}.
\end{equation*}
This is a linear subspace of $\R^n$ which
contains both $V^*(A)$ (as its convex subcone)
and $V(A)$ (as a max subcone of $V^*(A)$). If $B=X^{-1}AX$ with
$X=\diag(x)$ and $x$ positive, then, by Prop.~\ref{diag-sim-eq},
we have $C(A)=C(B)$, and we infer that $\Lin(C(A))=\{Xy\mid y\in\Lin(C(B))\}$.

\noindent

Let $K$ be a convex cone. The least linear space which contains $K$ will be
called the {\em linear hull} of $K$ and denoted by
$\lin(K)$. This is a special case of the affine hull of a convex set,
see \cite{Gru:67}. Denote by $B_y^{\varepsilon}$ the
open ball with radius $\varepsilon>0$ and centered at $y$.
The {\em relative interior} of $K$, denoted by $\ri(K)$,
is the set of points
$y\in\Rpn$ such that for
sufficiently small $\varepsilon$ we have that $B_y^{\varepsilon}\cap\lin(K)\subseteq K$.
If $\lin(K)=\R^n$, then it is the {\em interior} of $K$, denoted by $\intererer(K)$.

The following important ``splitting'' lemma can be deduced from \cite{Zie:94},
Lemma 2.9. 
\begin{lemma}
\label{split}
Suppose that $K\subseteq\Rpn$ is a convex cone which is a solution set of
a finite system of linear inequalities $S$. Let
$S_1$ be composed of the inequalities of $S$ which are satisfied by
all points in $K$ with equality, and $S_2:=S\backslash S_1$ be non-empty.
\begin{itemize}
\item[1.] There exists a point in $K$ by which
all inequalities in $S_2$ are satisfied strictly.
\item[2.] $\lin(K)$ is the solution set to $S_1$,
and $\ri(K)$ is the cone which consists of the points in $K$ by
which all inequalities in $S_2$ are satisfied strictly.
\end{itemize}
\end{lemma}

Now we describe all scalings that give rise to strict visualization.

\begin{theorem}
\label{relint2} Let $A\in\Rpnn$ and let $\lambda(A)>0$.
\begin{itemize}
\item[1.] $\Lin(C(A))$ is the linear hull of the subeigencone
$V^*(A)$.
\item[2.] $x\in\ri(V^*(A))$ if and only if, for $X=diag(x)$,
the matrix $X^{-1}AX$ is strictly visualized.
\item[3.] $\ri(V^*(A))$ contains the eigenvectors of $A$ if and only if
$V^*(A)=V(A)$.
\item[4.] If $A$ is definite, then any positive linear combination, and, if $A$
is irreducible, also any positive log-convex combination $x$ of all
columns of $A^*$ belongs to $\ri(V^*(A))$  and $X^{-1}AX$ with
$X=\diag(x)$ is strictly visualized.
\end{itemize}
\end{theorem}

\begin{proof}1. and 2.:
Consider Lemma \ref{split} with $K=V^*(A)$, then $V^*(A)$ is the solution set to the system
of inequalities $a_{ij}x_j\leq x_i$, and we need to show that the inequalities
with $(i,j)\in E_c(A)$, and those with $(i,j)\notin E_c(A)$,
play the role of $S_1$, and $S_2$ of Lemma \ref{split}, respectively.
For this, we note that by Theorem~\ref{subeigs} part~6., the inequalities with $(i,j)\in E_c(A)$
are satisfied with equality for all $x\in V^*(A)$, and Prop. \ref{fiedler-ptak2} implies that
there is $x\in V^*(A)$ by which all the inequalities with $(i,j)\notin E_c(A)$ are
satisfied strictly.

3.: The ``if'' part is obvious. The ``only if'' part: from Theorem \ref{subeigs} it follows
that $V^*(A)=V(A)$ if and only if the set of
critical nodes is $[n]$. Suppose that $V(A)$ is properly contained in $V^*(A)$,
then there is a node $i$ which is not critical. Then for any eigenvector $y$
there is an edge $(i,j)$ for which $a_{ij}y_j=y_i$
and obviously $(i,j)\notin E_c(A)$. Hence $y\notin\ri(V^*(A))$.

4.: Follows from Theorem~\ref{fiedler-ptak1} and part~2.
\end{proof}\eproof \\

Note that as $V^*(A)$ is the {\em max algebraic} column span
of $A^*$, its relative interior may also contain vectors which are
not positive linear combinations or positive log-convex combinations
of the columns of $A^*$. However, the relative interior of
$V^*(A)$, or the set of vectors
which lead to strict visualization, is exactly the set of vectors that
can be represented as positive combinations of all {\em convex} extremals
in $\ext(V^*(A))$, see \cite{Gru:67} Sect. 2.3.

We also remark here that a bijection between $\ri(V^*(A))$ and $\ri(V^*(A^T))$ is given
by $x\mapsto x^{-1}$, since $\lambda(A)=\lambda(A^T)$
and if $x$ is positive, then $a_{ij}x_j=\lambda(A) x_i$
(resp. $a_{ij}x_j<\lambda(A)x_i$)
holds if and only if $a_{ij}x_i^{-1}=\lambda(A) x_j^{-1}$ (resp. $a_{ij}x_i^{-1}<\lambda(A)x_j^{-1}$).
In particular, positive
linear combinations of {\em rows} of Kleene stars also lead, after the inversion,
to strict visualization scalings.

If $A$ is strongly definite (that is, $\lambda(A)=1$ and $a_{ii}=1$ for all $i\in[n]$),
then by Prop.~\ref{eig=subeig} we have $V^*(A)=V(A)$, so that $V(A)$ is convex and the maximum cycle
geometric mean can be strictly visualized by eigenvectors in $\ri(V(A))$.
We note that in the case when, in addition, the weights
of all non-trivial cycles are strictly less than $1$, the strict visualization scalings
have been described in \cite{But-03}.

Strongly definite matrices are related to the
{\em assignment problem.} By this we understand the following task: Given $%
A\in\Rpnn$ find a permutation $\pi
\in P_{n}$ such that its {\em weight}
$a_{1,\pi (1)}\cdot a_{2,\pi (2)}\cdot\ldots\cdot a_{n,\pi (n)}$
is maximal. A permutation $\pi$ of maximal weight will be also called
a {\em maximal} permutation.

Again, our aim is to precisely identify (``visualize'') matrix entries belonging to an optimal
solution using matrix scaling. That is, for a matrix $A$ with nonzero permutations, find diagonal
matrices $X$ and $Y$ such that all entries of $XAY$ on maximal permutations are equal to $1$
and that all other entries are strictly less than $1$.

To do this, we first find a maximal permutation $\pi$ and define the corresponding permutation
matrix $D^{\pi}$ by
\begin{equation*}
D^{\pi}_{ij}=
\begin{cases}
a_{ij}, & \text{if $j=\pi(i)$,}\\
0, & \text{if $j\neq\pi(i)$.}
\end{cases}
\end{equation*}
Using this matrix, we scale $A$ to one of its strongly definite forms $(D^{\pi})^{-1} A$.
In a strongly definite matrix, any maximal permutation is decomposed into critical cycles.
Conversely, any critical cycle can be extended to a maximal permutation, using the
diagonal entries.
Therefore, scalings $X$ which visualize the maximal permutations
of $(D^{\pi})^{-1} A$ are scalings which visualize
the critical cycles, and these are
given by Theorem \ref{relint2}.
After we have done this diagonal similarity scaling, we need
permutation matrix $E^{\pi^{-1}}=(\kroneker_{i\pi^{-1}(i)})$ to bring
all permutations again to their right place. Thus we get scaling $E^{\pi^{-1}} X^{-1} (D^{\pi})^{-1} A X$
which visualizes all maximal permutations.

Numerically, solving visualization problems by the methods described
above, relies on the following three standard problems: finding the
maximal cycle mean, computing the Kleene star of a matrix, and
finding a maximal permutation.  The first problem can be solved by
Karp's method \cite{BCOQ,karp,HOW:05}, the second problem can be
solved by the Floyd-Warshall algorithm \cite{PS} and the third problem
can be solved by the Hungarian method \cite{PS}. All of these methods
are polynomial and require $O(n^3)$ operations, which also gives a
complexity bound for the visualization problems.

Finally we note that the problem of strict visualization is related
to the problem of max balancing considered in
\cite{RSS:92,SS:90,SS:91}. A matrix $B$ is max balanced if and only
if each non-zero element lies on a cycle on which it is a minimal
element. It follows that $B$ is strictly visualized. It was shown in
\cite{RSS:92,SS:90,SS:91} that for each irreducible nonnegative $A$
there is an essentially unique diagonal matrix $X$ such that the scaling $B
= X^{-1}AX$ is max balanced, and hence there is a unique max
balanced matrix $\MB(A)$ diagonally similar to $A$. Importantly, the
matrix $\MB(A)$ is canonical for diagonal similarity of irreducible
nonnegative matrices, that is $A$ is diagonally similar to $C$ if
and only if $\MB(A) = \MB(C)$. A complexity bound for max balancing
which follows from \cite{RSS:92,SS:90,SS:91}, is $O(n^4)$, see also
\cite{YTO:06} for a faster version of the max balancing algorithm.

\section{Diagonal similarity scalings which leave a matrix visualized}

Another approach to describing the visualization scalings is to
start with a visualized matrix and describe all scalings which leave it visualized.

We first describe the Kleene star of a definite visualized matrix $A\in\Rpnn$. Let
$C^*(A)$ have $m$ strongly connected components $C_{\mu}$, where
$\mu\in[m]$, and denote by $N_{\mu}$ the set of nodes in $C_{\mu}$.
Denote by $A_{\mu\nu}$ the {\em
$(\mu,\nu)$-submatrix} of $A$ extracted from the rows with indices
in $N_{\mu}$ and from the columns with indices in $N_{\nu}$. Let
$A^C\in\Rpmm$ be the $m\times m$ matrix with entries
$\alpha_{\mu\nu}=\max\{a_{ij}\mid i\in N_{\mu},\, j\in
N_{\nu}\}$, and let $E\in\Rpnn$ be the $n\times n$ matrix with all
entries equal to $1$.

\begin{proposition}
\label{vis-kls}
Let $A\in\Rpnn$ be a definite visualized (resp. strictly visualized) matrix, let $m$
be the number of strongly connected components of $C^*(A)$ and let $A^C=(\alpha_{\mu\nu})$,
$A^*_{\mu\nu}$ and $E_{\mu\nu}$ be as defined above.
Then
\begin{itemize}
\item[1.] $\alpha_{\mu\mu}=1$ for all $\mu\in[m]$ and $\alpha_{\mu\nu}\leq 1$
(resp. $\alpha_{\mu\nu}<1$ for $\mu\neq\nu$), where
$\mu,\nu\in[m]$);
\item[2.] any $(\mu,\nu)$-submatrix of $A^*$ is equal to
$A^*_{\mu\nu}=\alpha_{\mu\nu}^* E_{\mu\nu}$, where $\alpha_{\mu\nu}^*$ is
the $(\mu,\nu)$-entry of $(A^C)^*$, and $E_{\mu\nu}$ is the
$(\mu,\nu)$-submatrix of $E$.
\end{itemize}
\end{proposition}
\begin{proof} 1.: Immediate from the definitions.

2.: Take any $i\in N_{\mu},$ $j\in N_{\nu}$, and any path
$\pi=(i_1,\ldots,i_k)$ with $i_1:=i$ and $i_k:=j$. Then $\pi$ can be
decomposed as $\pi=\tau_1\circ \sigma_1\circ\tau_2\circ\ldots
\circ\sigma_{l-1}\circ\tau_l$, where $\tau_i$, for $i\in[l]$, are
(possibly trivial) paths which entirely belong to some critical
component $C_{\mu_i}$, with $\mu_1:=\mu$ and
$\mu_l=\nu$, and $\sigma_i$, for $i\in[l-1]$, are edges between
the strongly connected components. Then $w(\pi,A)\leq w(\pi',A)$,
where
$\pi'=\tau'_1\circ\sigma'_1\circ\tau'_2\circ\ldots\circ\sigma'_{l-1}\circ\tau'_l$
is also a path from $i$ to $j$ such that $\tau'_i$ entirely belong
to the same critical components as $\tau_i$, and $\sigma'_i$ are
edges connecting the same critical components as $\sigma_i$, but
$w(\sigma'_i,A)=\max\{a_{ij}\mid i\in N_{\mu_i},\, j\in
N_{\mu_{i+1}}\}$ and $w(\tau'_i,A)=1$. Such a path exists, since
in a visualized matrix, there exists a path of weight $1$ between
any nodes in the same component of the critical digraph. Thus
$a^*_{ij}$ is the greatest weight over all such paths $\pi'$. As
$\pi'$ bijectively correspond to the paths in the weighted digraph
associated with $A^C$, the claim follows.
\end{proof}

Note that, after a convenient simultaneous permutation of rows and columns,
we have that if $A$ is a definite visualized matrix, then

\begin{equation}
\label{e:vis-kls}
A^*=
\begin{pmatrix}
E_{11} & \alpha^*_{12} E_{12} &\ldots & \alpha^*_{1n} E_{1m}\\
\alpha^*_{21} E_{21} & E_{22} &\ldots & \alpha^*_{2n} E_{2m}\\
\vdots &\vdots &\ddots &\vdots\\
\alpha^*_{m1} E_{m1} & \alpha^*_{m2} E_{m2} &\ldots & E_{mm}
\end{pmatrix}.
\end{equation}

Note that $A^C$ does not contain critical cycles except for the loops,
otherwise $C_{\mu}$ are not the components of $C^*(A)$. Hence
$\Lin(A^C)=\R^m$, and we can speak of the interior of $V^*(A^C)$.

Given a strictly visualized matrix $A$ as above,
denote by $I_{\mu},\ \mu\in[m],$ the matrix such that $(I_{\mu})_{ij}=1$ whenever
$i=j$ belongs to $N_{\mu}$ and $(I_{\mu})_{ij}=0$ elsewhere, and by $A\dotplus B$
the direct sum of matrices $A$ and $B$.


\begin{theorem}
\label{theo:hs} Let $A\in\Rpnn$ be a definite visualized matrix and
let $m$ be the number of strongly connected components of $C^*(A)$.
Let $A^C$ and $I_{\mu}$ be as defined above. Then $X^{-1}AX$, where
$X=\diag(x)$ with $x\in\Rpn$ positive, is visualized (resp. strictly
visualized) if and only if $X$ has the form
\begin{equation*}
\label{eq:hs1}
X = \tilde{x}_1 I_1 \dotplus \cdots \dotplus \tilde{x}_m I_m,
\end{equation*}
where $\tilde{x}$ is a vector satisfying
$\alpha_{\mu\nu} \tilde{x}_{\nu}\leq \tilde{x}_{\mu}$ (resp.
$\alpha_{\mu\nu} \tilde{x}_{\nu}<\tilde{x}_{\mu}$), where $\mu\neq\nu$,  
$\mu,\nu\in[m]$.
In other words, $\tilde{x}\in V^*(A^C)$ (resp. $\tilde{x}\in\intererer(V^*(A^C))$).
\end{theorem}
\begin{proof}
The ``if'' part: Let $x$ be as described, then the 
elements $a_{ij}$, for $i,j\in N_{\mu}$,
do not change after the scaling, so each block $A_{\lambda\lambda}$ remains 
unchanged, and hence
visualized (resp. strictly visualized). For $a_{ij}$ with $i\in N_{\mu}$, $j\in N_{\nu}$,
$\mu\ne\nu$, we have that $a_{ij}x_j\leq x_i$ (resp. $a_{ij}x_j<x_i$),
as $x_i=\tilde{x}_{\mu}$, $x_j=\tilde{x}_{\nu}$, and
$\alpha_{ij}$ is the maximum over these $a_{ij}$. Hence
$X^{-1}AX$
is visualized (resp. strictly visualized).

The ``only if'' part:
Suppose that scaling by $X$ leaves
$A$ visualized (resp. makes $A$ strictly visualized).
As $A$ is initially visualized, all
critical edges have weights equal to $1$, and
$x$ should be such that $x_i=x_j=\tilde{x}_{\mu}$ whenever
$i,j$ belong to the same $N_{\mu}$.
For $i\in N_{\mu}$, $j\in N_{\nu}$, $\mu\neq\nu$,
we should have that $a_{ij} \tilde{x}_{\nu}\leq \tilde{x}_{\mu}$ (resp. $a_{ij} \tilde{x}_{\nu}<
\tilde{x}_{\mu}$).
Taking maximum over these $a_{ij}$, we obtain that
this is equivalent to
$\alpha_{\mu\nu} \tilde{x}_{\nu}\leq \tilde{x}_{\mu}$
(resp. $\alpha_{\mu\nu} \tilde{x}_{\nu}<
\tilde{x}_{\mu}$).

It remains to apply Lemma \ref{split} (with $S_1=\emptyset$), to obtain that the same
is equivalent to $\tilde{x}\in V^*(A^C)$ (resp. $\tilde{x}\in\intererer(V^*(A^C))$).
\end{proof}

In the following we discuss some issues concerning linear algebraic properties of Kleene cones
and Kleene stars. In this context,
Kleene stars are known as path product matrices, see
\cite{JS:99,JS-07,JS-08}.

For a matrix $A\in\Rpnn$ with $\lambda(A)>0$, we proved that
\begin{equation}
\label{eq:pca1}
\Lin(C(A))=\{x\in\Rn\mid a_{ij} x_j=\lambda(A)x_i,\ (i,j)\in E_c(A)\}.
\end{equation}
is the linear hull of $V^*(A)$. Note that in the case when $A$ is
definite and strictly visualized, $a_{ij}=1$ for all $(i,j)\in E_c(A)$
and $\lambda(A)=1$. Also see Section 2 for the definition of
$n(C(A))$ and $|\overline{N_c(A)}|$.

\begin{proposition}
\label{lina-eq}
Let $A\in\Rpnn$ have $\lambda(A)>0$.
\begin{itemize}
\item[1.] The dimension of $\Lin(C(A))$ is equal to the number of strongly connected
components in $C^*(A)$, that is, to $n(C(A))+|\overline{N_c(A)}|$;
\item[2.] If $A$ is definite, then $C^*(A)=C(A^*)$ and $\Lin(C(A))=\Lin(C(A^*))$.
\end{itemize}
\end{proposition}
\begin{proof}
Let $N_{\mu}$, for $\mu\in [m]$ where $m=n(C(A))+|\overline{N_c(A)}|$,
be the set of nodes of $C_{\mu}$, a strongly connected component of $C^*(A)$.
In the case when $A$ is
definite and strictly visualized, $C^*(A)=C(A^*)$ is seen from \eqref{e:vis-kls}, where
$\alpha^*_{\mu\nu}<1$ for all $\mu\neq\nu$, and it is also seen from \eqref{e:vis-kls}
that $\Lin(C(A^*))$ is the linear space
comprising all vectors $x\in\Rpn$ such that $x_i=x_j$ whenever $i$ and $j$ belong to the same
$N_{\mu}$. As $\Lin(C(A))$ is also equal to that space by
\eqref{eq:pca1}, we have that $\Lin(C(A))=\Lin(C(A^*))$. We can take, as a basis
of this space, the vectors $e^{\mu}$, for $\mu\in[m]$, such that
$e^{\mu}_j=1$ if $j\in N_{\mu}$ and $e^{\mu}_j=0$ if $j\notin N_{\mu}$,
and hence the dimension of $\Lin(C(A))$ is $n(C(A))+|\overline{N_c(A)}|$.
The general case can be obtained using diagonal similarity.
\end{proof}\eproof

Prop. \ref{lina-eq} enables us to present the following result.

\begin{theorem}
\label{sd-dim} For any matrix $A$ with $\lambda(A)>0$, the max
algebraic dimension of $V^*(A)$ is equal to the (linear algebraic)
dimension of $\Lin(C(A))$, which is the linear hull of $V^*(A)$.
\end{theorem}

\begin{proof}
It follows from Prop.~\ref{maxdim} and Prop.~\ref{lina-eq} part 1. that
both dimensions are equal to the number of strongly connected components in $C(A)$.
\end{proof}\eproof

When $A$ is strongly definite and the weights of all nontrivial
cycles are strictly less than $1$, Theorem \ref{sd-dim} implies that
$V^*(A)$ contains $n$ linearly independent vectors. This result has
been obtained by Butkovi\v{c} \cite{But}, Theorem 4.1. One could
also conjecture that in this case the columns of $A^*$ should be
linearly independent in the usual sense. However, this is not so in
general as we show by modifying Example 3.11 in Johnson-Smith
\cite{JS:99}. Let
\begin{equation*}
A = A^*=
\begin{pmatrix}
1 & 5/11 & 5/11 & 7/11 & 7/11 & 7/11 \\
5/11 & 1 & 5/11 & 7/11 & 7/11 & 7/11 \\
5/11 & 5/11 & 1 & 7/11 & 7/11 & 7/11 \\
7/11 & 7/11 & 7/11 & 1 & 5/11 & 5/11 \\
7/11 & 7/11 & 7/11 & 5/11 & 1 & 5/11 \\
7/11 & 7/11 & 7/11 & 5/11 & 5/11 & 1
\end{pmatrix}.
\end{equation*}
Then the linear algebraic rank of $A^*$ is $5$, however, by Theorem \ref{sd-dim} (or \cite{But},
Theorem 4.1) the max algebraic dimension of $V^*(A)$, and therefore
the linear algebraic dimension of $L(C(A))$, are 6. We observe that $x
= [7/11,7/11,7/11,1,1,1]^T$ is a max eigenvector of $A^*$ (hence in
$V^*(A)$) but it is not in the linear algebraic span of the columns of $A^*$. Finally we note that the
original form of Example 3.11 in \cite{JS:99} provides a Kleene star
with negative determinant.


\begin{thebibliography}{99}



\bibitem{AGW-05}
M. Akian, S. Gaubert and C. Walsh.
Discrete max-plus spectral theory.
In Idempotent Mathematics and Mathematical Physics (G. Litvinov and V. Maslov, eds.),
Contemporary Mathematics 377, AMS, Providence, 2005, pages 53--77.
E-print arXiv:math/0405225.

\bibitem{A:63} S.N. Afriat.
The system of inequalities $a\sb{rs}>X\sb{r}-X\sb{s}$.
  Proc. Cambridge Philos. Soc.  59 (1963), 125--133.

\bibitem{A:74} S.N. Afriat. On sum-symmetric matrices. Linear
 Algebra and Appl.  8  (1974), 129--140.

\bibitem{BCOQ} F.L. Baccelli, G. Cohen, G.-J. Olsder and J.-P.
Quadrat. Synchronization and Linearity. John Wiley, Chichester, New York,
1992.

\bibitem{Bap-98} R.B. Bapat. A max version of the Perron-Frobenius Theorem.
Lin.Alg. and Appl. 275/276 (1998), 3-18.

\bibitem{BP} A. Berman and R.J. Plemmons. Nonnegative matrices in the
mathematical sciences. Academic Press, 1979.





\bibitem{But} P. Butkovic. Simple image set of $(\max,+)$-linear mappings. Discrete
Applied Mathematics 105 (2000), 73-86.

\bibitem{But-03} P. Butkovic. Max-algebra: the linear algebra of
combinatorics? Lin. Alg. and Appl. 367 (2003), 313-335.

\bibitem{BS} P. Butkovic and H. Schneider. Applications of Max-algebra to
diagonal scaling of matrices. Electr. Lin. Alg. 13 (2005),
262-273.

\bibitem{BS-07} P. Butkovic and H. Schneider. On the visualization scaling of matrices.
Preprint 2007/12, Univ. of Birmingham, 2007.

\bibitem{BSS-07} P. Butkovic, H. Schneider and S. Sergeev.
Generators, extremals and bases of max cones. Lin. Alg. and Appl. 421 (2007), 394-406.

\bibitem{Car-71} B.A. Carr\'{e}. An algebra for network routing problems.
J. of the Inst. of Maths. and Applics. 7 (1971), 273-299.


\bibitem{CG:79} R.A. Cuninghame-Green. Minimax Algebra. Lecture Notes
in Economics and Mathematical Systems 166, Berlin, Springer, 1979.

\bibitem{CG:95} R.A. Cuninghame-Green. Minimax Algebra and Applications. In:
Advances in Imaging and Electron Physics, vol. 90, pp. 1--121, Academic
Press, New York, 1995.

\bibitem{DS-04} M. Develin and B. Sturmfels.  Tropical convexity. Documenta Math.
9 (2004), 1--24. E-print arXiv:math/0308254.

\bibitem{ED-99} L. Elsner and P. van den Driessche. On the power method in max algebra.
Lin. Alg. and Appl. 302-303 (1999), 17-32.

\bibitem{ED-01} L. Elsner and P. van den Driessche. Modifying the power method in max
algebra. Lin. Alg. and Appl. 332-334 (2001), 3-13.

\bibitem{ES} G.M. Engel and H. Schneider. Cyclic and diagonal products on a matrix.
Lin. Alg. and Appl. 7 (1973), 301-335.

\bibitem{ES:75} G.M. Engel and H. Schneider. Diagonal similarity and diagonal
 equivalence for matrices over groups with 0.
 Czechoskovak. Math. J., 25(100) (1975), 387-403.

\bibitem{FP:67} M. Fiedler and V. Pt\'{a}k. Diagonally dominant matrices.
Czechoslovak Math. J. 17(92)  (1967), 420--433.

\bibitem{FP} M. Fiedler and V. Pt\'{a}k.
Cyclic products and an inequality for determinants. Czechoslovak
Math. J., 19(94) (1969), 428-450.

\bibitem{Gau:92} S. Gaubert. Th\'{e}orie des syst\`{e}mes lin\'{e}%
aires dans les dio\"{\i}des. Th\`{e}se, Ecole des Mines de Paris, 1992.

\bibitem{Gau-95} S. Gaubert. Resource Optimization and ($\min,+$)
Spectral Theory. IEEE Trans. on Automat. Control, 40(11) (1995) 1931-1934.

\bibitem{GK-07}
S. Gaubert and R. Katz. The Minkowski theorem for max-plus convex sets.
Linear Algebra Appl. 421 (2007), 356-369. E-print arXiv:math/0605078.

\bibitem{Gru:67} B. Gr\"{u}nbaum. Convex polytopes. Wiley, 1967.

\bibitem{HOW:05} B. Heidergott, G.J. Olsder and J. van der Woude.
Max Plus at Work: Modeling and Analysis of Synchronized Systems, A
Course on Max-Plus Algebra. Princeton Univ. Press, 2006.

\bibitem{HS} D. Hershkowitz and H. Schneider.
One sided simultaneous inequalities and sandwich theorems for
diagonal similarity and diagonal equivalence of nonnegative
matrices. Electr. Lin. Alg. 10 (2003), 81 - 101.

\bibitem{JS:99} C.R. Johnson and R.L. Smith.
Path product matrices. Linear and Multilinear Algebra 46 (1999),
177--191.

\bibitem{JS-07} C.R. Johnson and R.L. Smith.
Positive, path product and inverse M-matrices. Linear Algebra Appl.
421 (2007) 328--337 and 423 (2007) 519.

\bibitem{JS-08} C.R. Johnson and R.L. Smith.
Path product matrices and eventually inverse M-matrices. SIAM J.Matrix Anal. Appl.
29(2) (2008) 370-376.

\bibitem{Jos-05} M. Joswig. Tropical halfspaces.
In Combinatorial and computational geometry (J.E. Goodman, J.~Pach, and E.~Welzl, eds.),
MSRI publications 52, Cambridge Univ. Press, 2005, pages 409--432.
E-print arXiv:math/0312068.

\bibitem{JK-08} M. Joswig and K. Kulas. Tropical and ordinary convexity combined, 2008.
E-print arXiv:0801.4835.

\bibitem{karp} R. M. Karp. A characterization of the minimum cycle mean in a
digraph. Discrete Mathematics 23 (1978), 309-311.


\bibitem{PS} C.H. Papadimitriou, K. Steiglitz. Combinatorial Optimization: Algorithms and
Complexity. Prentice Hall, New Jersey, 1982.

\bibitem{RSS:92} U.G. Rothblum, H. Schneider and M.H. Schneider.
 Characterizations of max-balanced flows.  Disc. Appl. Math. 39 (1992), 241-261.

\bibitem{RSS} U.G. Rothblum, H. Schneider and M.H. Schneider.
Scaling matrices to prescribed row and column maxima.
SIAM J. Matrix Anal. Appl. 15 (1994), 1-14.

\bibitem{SS:90} H. Schneider and M.H. Schneider.
Towers and cycle covers for max-balanced graphs.
 Congress Num. 78 (1990), 159-170.

\bibitem{SS:91} H. Schneider and M.H. Schneider. Max-balancing
weighted directed graphs. Math. Oper. Res. 16 (1991), 208-222.

\bibitem{Ser-09} S. Sergeev. On cyclic classes and attraction spaces in max algebra, 2009.
E-print arXiv:0903.3960. 

\bibitem{Vor-67} N.N. Vorobyov. Extremal Algebra of Positive Matrices.
Elektronische Informationsverarbeitung und Kybernetik 3 (1967),
39-71 (in Russian).

\bibitem{YTO:06} N.E. Young, R.E. Tarjan, J.B. Orlin. Faster
parametric shortest path and minimumn-balance algorithm. Networks 21
(2006), 205-221.

\bibitem{Zie:94} G.M. Ziegler. Lectures on Polytopes. Springer, 1994.



\end{thebibliography}
\end{document}